\documentclass[letterpaper, 10 pt, conference]{ieeeconf}  % Comment this line out if you need a4paper

\IEEEoverridecommandlockouts                              % This command is only needed if 
\usepackage{amsfonts}
\usepackage{amssymb}  % assumes amsmath package installed
\usepackage{amsmath}
\usepackage{bm}
\usepackage{tikz}
\usepackage{pgfplots}
\newlength\breite
\newlength\hohe
\usepackage{float} % for \newfloat{algorithm}
\usepackage{mathtools}

\usepackage[noadjust]{cite} %IEEE standard

\newfloat{algorithm}{t}{}

\newcommand{\norm}[1]{\left\lVert#1\right\rVert}
\newcommand\numberthis{\addtocounter{equation}{1}\tag{\theequation}}
\renewcommand{\refeq}[1]{\stackrel{\mathclap{#1}}{=}}
\newcommand{\refleq}[1]{\stackrel{\mathclap{#1}}{\leq}}

\newcommand{\grad}{\nabla}

\newtheorem{lemma}{Lemma}
\newtheorem{remark}[lemma]{Remark}
\newtheorem{assumption}[lemma]{Assumption}
\newtheorem{theorem}[lemma]{Theorem}

\title{\LARGE \bf
An online convex optimization algorithm for controlling linear systems with state and input constraints
}

\author{Marko Nonhoff$^{1}$ and Matthias A. M\"uller$^{1}$%
	\thanks{$^{1}$Leibniz University Hannover, Institute of Automatic Control, 30167 Hannover, Germany.	Email: \{nonhoff,mueller\}@irt.uni-hannover.de}%
}

\newcommand\copyrighttext{%
	\footnotesize \copyright 2021 IEEE. Personal use of this material is permitted. Permission from IEEE must be obtained for all other uses, in any current or future media, including reprinting/republishing this material for advertising or promotional purposes, creating new collective works, for resale or redistribution to servers or lists, or reuse of any copyrighted component of this work in other works.}
\newcommand\copyrightnotice{%
	\begin{tikzpicture}[remember picture,overlay]
		\node[anchor=south,yshift=10pt] at (current page.south) {\fbox{\parbox{\dimexpr\textwidth-\fboxsep-\fboxrule\relax}{\copyrighttext}}};
	\end{tikzpicture}%
}

\allowdisplaybreaks

\begin{document}

\maketitle
\thispagestyle{empty}
\pagestyle{empty}
\copyrightnotice

%%%%%%%%%%%%%%%%%%%%%%%%%%%%%%%%%%%%%%%%%%%%%%%%%%%%%%%%%%%%%%%%%%%%%%%%%%%%%%%%
\begin{abstract}
	
	This paper studies the problem of controlling \linebreak linear dynamical systems subject to point-wise-in-time constraints. We present an algorithm similar to online gradient descent, that can handle time-varying and a priori unknown convex cost functions while restraining the system states and inputs to polytopic constraint sets. Analysis of the algorithm's performance, measured by dynamic regret, reveals that sublinear regret is achieved if the variation of the cost functions is sublinear in time. Finally, we present an example to illustrate implementation details as well as the algorithm's performance and show that the proposed algorithm ensures constraint satisfaction.
	
\end{abstract}

%%%%%%%%%%%%%%%%%%%%%%%%%%%%%%%%%%%%%%%%%%%%%%%%%%%%%%%%%%%%%%%%%%%%%%%%%%%%%%%%
\section{INTRODUCTION} \label{sec:Introduction}

Application of methods from online learning and optimization leads to new techniques for learning-based controller synthesis. In this paper, we apply the online convex optimization (OCO) framework first introduced in \cite{Zinkevich2003} to the problem of controlling constrained linear dynamical systems. OCO is an online variant of classical numerical optimization, where the cost function to be minimized is time-varying and a priori unknown. Specifically, at every time $t$, an algorithm chooses an action $y_t \in \mathbb Y$ from a convex constraint set $\mathbb Y$ based on the chosen actions in the past and the observed cost functions. After the action $y_t$ is chosen, the environment reveals a new cost function $L_t: \mathbb Y \rightarrow \mathbb R$ which leads to the cost $L_t(y_t)$. The goal is to minimize the total cost $\sum_{t=1}^T L_t(y_t)$ over $T$ stages. This problem has been studied extensively in the online optimization and learning community (see \cite{Shwartz2012, Hazan2016} for an overview) with a focus on the non-asymptotic performance of the algorithms. Typically, dynamic regret is chosen as a performance measure, which is defined as the cumulative gap between the cost observed by the algorithm and some benchmark \cite{Jadbabaie2015, Besbes2015, Mokhtari2016}. A sublinear regret bound is desirable, implying that the algorithm's performance is asymptotically on average no worse than the benchmark. This framework enjoys advantages in its ability to handle time-varying unknown cost functions while ensuring constraint satisfaction and the low computational complexity of its algorithms, which are desirable in controller synthesis, too.

In the classical OCO framework, one typically considers that any action $y_t \in \mathbb Y$ can be chosen. Thereby, no underlying dynamical system can be considered. Application of OCO to the control of dynamical systems has already been studied by introducing a switching cost or ramp cost $d(y_t-y_{t-1})$ to study the effect of a time coupled cost function \cite{Tanaka2006}. In particular, in \cite{Li2018}, a switching cost which can be seen as an additional quadratic cost on the input $u_{t-1}$ of a single integrator system $y_t = y_{t-1} + u_{t-1}$ is studied, which is then extended to the case of general linear systems in \cite{Li2019}. A similar approach is taken in \cite{Nonhoff2020}, where a sublinear regret bound for general controllable linear systems is derived. In \cite{Abbasi2014,Cohen2018,Akbari2019}, linear dynamical systems subject to quadratic cost functions are considered and it is shown, that the regret with respect to the best linear controller is sublinear. Therein, the algorithms update a linear control strategy at every time step. This approach is extended in \cite{Agarwal2019} to general convex cost functions. Whereas in the classical OCO framework the allowed actions $y_t$ are typically restricted to a constraint set $\mathbb Y$, none of these previous works on the combination of OCO and dynamical systems considers state or input constraints. A similar setting is considered in \cite{Colombino2020,Bianchin2020,Hauswirth2021}, where online optimization is employed to control the output of an exponentially stable system to the steady-state solutions of an optimization problem. While constraints on the solutions of the optimization problems can be considered, point-wise-in-time constraints on both the input and state trajectories of the dynamical system are not. This approach is extended to stabilizable LTI systems in \cite{Lawrence2018}.

This work builds on and extends the results in \cite{Nonhoff2020} such that constraints on both the inputs and states of the system are guaranteed to be satisfied at all times. Ensuring satisfaction of these constraints is of paramount importance, e.g., in safety-critical applications or where only a finite amount of control energy is available. As discussed above, ensuring constraint satisfaction is not possible in the existing results in the literature and it requires substantial modifications in the algorithm design and the corresponding analysis. Our analysis reveals that the proposed algorithm still enjoys sublinear regret, as is the case for existing algorithms which cannot guarantee satisfaction of input and state constraints. 

This paper is organized as follows. Section~\ref{sec:Setting} defines the problem setting, whereas our algorithm is proposed and discussed in Section~\ref{sec:Algorithm}. We proceed to give a regret analysis and our main theorem in Section~\ref{sec:Regret}. In Section~\ref{sec:Simulations}, we illustrate implementation details and the performance of the proposed algorithm. Section~\ref{sec:Conclusion} concludes the paper.

\textit{Notation}: For a vector $x \in \mathbb{R}^n$, $x_i$ denotes its $i$-th entry and $\norm{x}$ the Euclidean norm, whereas for a matrix $A \in \mathbb{R}^{n\times m}$, $A_i$ is its $i$-th row and $\norm{A}$ the corresponding induced matrix norm. Given a set $\mathcal S \subset \mathbb R^n$, $\Pi_{\mathcal S}(x) = \arg \min_{s \in \mathcal S} \norm{x-s}^2$ is the projection of $x\in\mathbb R^n$ onto the set $\mathcal S$. We define by $\mathbb{N}_{[a,b]}$ the set of natural numbers in the interval $[a,b]$. The gradient of a function $f(x)$ evaluated at $x$ is denoted by $\nabla f(x)$. Additionally, $I_n$ is the identity matrix of size $n \times n$.

\section{SETTING} \label{sec:Setting}

We consider discrete-time linear systems of the form
\begin{align}
x_{t+1} &= Ax_t + Bu_t, &&x(0) = x_0, \label{eq:SysDyn}
\end{align}
where $x_t \in \mathbb R^n$ are the states of the system and $u_t \in \mathbb R^m$ are the control inputs. The matrices $A \in \mathbb R^{n \times n}$ and \mbox{$B \in \mathbb R^{n\times m}$} are assumed to be known. System~\eqref{eq:SysDyn} is subject to state constraints $x_t \in \mathcal X$ and input constraints $u_t \in \mathcal U$ which have to be satisfied at every time instant $t \in \mathbb N_{[1,T]}$. We assume both constraint sets $\mathcal X$ and $\mathcal U$ to be compact polytopes.

\begin{assumption} \label{assump:constraints}
	The state and input constraint sets are compact convex polytopes with $0$ in their interior, given by $\mathcal X = \{x \in \mathbb R^n | C_x x \leq d_x \}$ and $\mathcal U = \{u \in \mathbb R^m | C_u u \leq d_u \}$, where $C_x \in \mathbb R^{c_x \times n}$, $d_x \in \mathbb R^{c_x}$, $C_u \in \mathbb R^{c_u \times m}$, and $d_u \in \mathbb R^{c_u}$.
\end{assumption}

Note that compactness of $\mathcal U$ implies existence of a finite constant $D_u$ such that $\norm{u_1 - u_2} \leq D_u$ for all $u_1,u_2 \in \mathcal U$.

The online control problem under consideration in this work is described as follows: At every time step $t \in \mathbb N_{[1,T]}$, the current state $x_t$ is measured. The controller has to decide a control input $u_t \in \mathcal U$, based only on the measured state and previous cost functions, and apply the control input $u_t$ to system~\eqref{eq:SysDyn}. Afterwards, the cost function $L_t: \mathcal X \times \mathcal U \rightarrow \mathbb R$ is revealed by the environment resulting in the cost $L_t (x_t,u_t)$. Finally, the system evolves to the next state $x_{t+1}$. The goal is to minimize the total cost over $T$ stages. 

As common in OCO, we consider regret as a measure for our algorithm's performance. The optimal state and input sequence $\bm x_t^* = \{ x_1^*,~\dots,~x_T^* \}$ and $\bm u_t^* = \{ u_1^*,~\dots,~u_T^* \}$, respectively, are the solution to the optimization problem
\begin{equation*}
\min_{\bm x \in \mathcal X^T, \bm u \in \mathcal U^T} \sum_{t=0}^T L_t(x_t,u_t) \quad \text{s.t. } x_{t+1} = Ax_{t} + Bu_t.
\end{equation*}
Thus, $(x^*_t,u^*_t)$ denote the optimal states and inputs at time instant $t$ in hindsight, when all cost functions $L_t$ are known. Then, in our case, we define the dynamic regret $\mathcal R$ as
\vspace{-2pt}
\begin{align} \label{eq:regret}
	\mathcal R &= \sum_{t=0}^T L_t(x_t,u_t) - L_t(x_t^*,u_t^*).
\end{align}
The dynamic regret $\mathcal R$ measures how much performance is lost due to not knowing the cost functions $L_t$ a priori. The definition in \eqref{eq:regret} is in line with the dynamic regret measure imposed in \cite{Li2019}. Another popular regret measure is comparing the algorithm's performance to the best linear feedback controller \cite{Abbasi2014, Cohen2018, Akbari2019, Agarwal2019}, which is a weaker benchmark since the optimal trajectories may not result from a linear feedback. Next, we require some technical assumptions, which are fairly standard in OCO (compare \cite{Mokhtari2016, Li2018, Li2019, Nonhoff2020}).

\begin{assumption} \label{assump:costfunctions}
	For every $t \in \mathbb N_{[0,T]}$, the cost function $L_t$ satisfies
	\begin{enumerate}
		\item $L_t(x,u) = f_t^x(x) + f_t^u(u)$,
		\item $f_t^x(x)$ is $\alpha_x$-strongly convex, $l_x$-smooth\footnote{See \cite{Nesterov2018} for a definition of $\alpha$-strong convexity and $l$-smoothness.} for all $x \in \mathcal X$,
		\item $f_t^u(u)$ is $\alpha_u$-strongly convex, $l_u$-smooth for all $u \in \mathcal U$.
	\end{enumerate}
\end{assumption}

Note that Lipschitz continuity of the cost functions \linebreak $f_t^x: \mathcal X \rightarrow \mathbb R$ and $f_t^u: \mathcal U \rightarrow \mathbb R$ with Lipschitz constants $L_x$ and $L_u$ follows from $l$-smoothness and compactness of the constraint sets $\mathcal X$ and $\mathcal U$, respectively.

Additionally, we define $\theta_t = \arg \min_{x \in \mathcal X} f_t^x(x)$ and $\eta_t = \arg \min_{u \in \mathcal U} f_t^u(u)$. Note that due to compactness of the sets $\mathcal X$ and $\mathcal U$ and strong convexity of the cost functions, the minima are attained, finite, and unique. In contrast to the trajectories $\bm x^*$ and $\bm u^*$, the sequences $\bm{\theta} = \{ \theta_1,~\dots,~\theta_T \}$ and $\bm \eta = \{ \eta_1,~\dots,~\eta_T \}$ in general do \emph{not} satisfy the system dynamics~\eqref{eq:SysDyn}. If the cost functions $L_t$ are allowed to change at every time instant, we will not be able to achieve low dynamic regret. Therefore, we define a measure for the variation of the cost functions similar to \cite{Mokhtari2016, Li2018} as $\text{Path length} := \sum_{t=0}^T \norm{\theta_t - \theta_{t-1}} + \sum_{t=0}^T \norm{\eta_t - \eta_{t-1}}$. We further restrict the class of cost functions by only considering tracking setpoints of system~\eqref{eq:SysDyn}. Let $\bar{\mathcal X} = \{ x \in \mathbb R^n | C_x (x+\delta r) \leq d_x~\forall~\norm{r}\leq1 \}$, where $\delta>0$.
\begin{assumption} \label{assump:tracking}
	For all $t \in \mathbb N_{[0,T]}$, $\theta_t$ and $\eta_t$ satisfy $\theta_t \in \bar{\mathcal X}$, $\eta \in \mathcal U$, and $\theta_t = A\theta_t + B\eta_t$.
\end{assumption}

Assumption~\ref{assump:tracking} states that the minimum $(\theta_t,\eta_t)$ of the cost function $L_t(x_t,u_t)$ at time instant $t$ is a feasible steady state with respect to the system dynamics~\eqref{eq:SysDyn} and the constraints. Hence, the control objective is to track a priori unknown and time-varying setpoints. Relaxing this assumption to general convex cost functions (termed \emph{economic} cost functions in the context of model predictive control (MPC) \cite{Faulwasser2018}) is part of our ongoing work. Moreover, Assumption~\ref{assump:tracking} restricts the optimal states $\theta_t$ to the interior of the constraint set $\mathcal X$. It is straightforward to show that the shrinked set $\bar{\mathcal X}$ can equivalently expressed as the polytope $\bar{\mathcal X} = \{ x \in \mathbb R^n | C_x x \leq \bar d_x \}$, where $\bar d_x \in \mathbb R^{n_c}$ is defined element-wise by $\bar d_{x,i} = d_{x,i} - \delta \norm{C_{x,i}}$. 
%Note that the cost functions $L_t$ and the corresponding optimizers $(\theta_t, \eta_t)$ are only defined for $t \in \mathbb N_{[1,T]}$. Hence, we let without loss of generality $L_0(x,u) = f_0^x(x) + f_0^u(u)$ such that Assumption~\ref{assump:costfunctions} is satisfied, $\theta_0 = \arg \min_{x \in \mathbb R^n} f_0^x(x) \in \bar{\mathcal X}$ and $\eta_0 = \arg \min_{u \in \mathbb R^m} f_0^u(u) \in \mathcal U$. A convenient choice for the values of $\theta_0$ and $\eta_0$ is given in Section~\ref{sec:Algorithm}.

Similar to \cite{Nonhoff2020}, we assume system~\eqref{eq:SysDyn} to be controllable and $\norm{A}$ to be bounded as stated in Assumption~\ref{assump:system}.
\begin{assumption} \label{assump:system}
	The pair $(A,B)$ is controllable, i.e.,
	\begin{align*}
	\text{rank} \begin{pmatrix} B & AB & \dots & A^{n-1}B \end{pmatrix} = n
	\end{align*}
	and $\norm{A} < \frac{l_x + \alpha_x}{l_x-\alpha_x}$.
\end{assumption}
As discussed in~\cite{Nonhoff2020}, a bound on $\norm{A}$, which can be seen as a bound on the instability of system \eqref{eq:SysDyn}, is necessary since we want to control the system by applying one gradient descent step at every time instant $t$. Therefore, one gradient descent step needs to be able to counteract the instability of the system, which yields Assumption~\ref{assump:system}. It can also be seen that, if $\alpha_x = l_x$, which is the case for, e.g., $f_t^x(x) = \norm{x_t-\theta}^2$ for some $\theta \in \mathbb R^n$, then any controllable system satisfies Assumption~\ref{assump:system}. Finally, we require that any state in $\mathcal X$ can be reached from every initial state in $\mathcal X$ in finite time. 
\begin{assumption} \label{assump:contrconst}
	There exists a constant $\mu \in \mathbb N$ such that for every two states $x, y \in \mathcal X$, there exists a feasible input trajectory $\bm u = \{ u^{(1)},~\dots,~u^{(\mu)}  \}$ satisfying $A^\mu x + S_c u = y$, where $S_c = \begin{pmatrix} B & AB & \dots & A^{\mu-1}B \end{pmatrix}$.
\end{assumption}
An input trajectory $\bm u = \{u_1,\dots,u_\tau \}$, $\tau \in \mathbb N$ is called feasible if it satisfies both the input and state constraints, i.e., $u_t \in \mathcal U$ and $x_t \in \mathcal X$ when applying $\bm u$ for all $t \in \mathbb{N}_{[1,\tau]}$. Assumption~\ref{assump:contrconst} can be interpreted as assuming controllability under constraints. If Assumption~\ref{assump:contrconst} is not satisfied for a state constraint set ${\mathcal X}_0$ and an input constraint set $\mathcal U$, a suitable subset $\mathcal X$ of the viability kernel\footnote{See \cite{Aubin2011} for a definition of the viability kernel and an overview of viability theory. See \cite{Boccia2014} for an application of viability theory to MPC.} has to be found that renders Assumption~\ref{assump:constraints} and Assumption~\ref{assump:contrconst} satisfied. 

\begin{remark} \label{rem:PredHor}
	Whereas Assumption~\ref{assump:contrconst} itself is natural in our setting, we assume the constant $\mu$ to be known in Algorithm~1. This potentially leads to a large prediction horizon and degrading performance, see Section~\ref{sec:Algorithm} for details. The question how Algorithm~1 needs to be modified in order to shorten the prediction horizon while maintaining a sublinear regret bound is an interesting problem for future research.
\end{remark}

\section{ALGORITHM} \label{sec:Algorithm}

\begin{algorithm} %\centering
	\vspace{5pt}
	{
		\setlength\belowdisplayskip{0pt}
		\setlength\abovedisplayskip{0pt}
		\fbox{\parbox{.95\linewidth}{
				
				\textbf{Algorithm 1}~(OGD for constrained linear systems)
				
				\vspace{-6pt}						
				\rule{.42\textwidth}{.5pt}
				
				Given step sizes $\gamma_v$ and $\gamma_x$, initialization $\bm{\hat u_0}$, $v_0$, and measured state vector $x_{t-1}$.
				
				At time $t \in [1,T]$:
				\begin{flalign}
					&v_t = \Pi_{\mathcal U} (v_{t-1} - \gamma_u \nabla f_{t-1}^u(v_{t-1})) \label{algo:OGDInputs} &\\
					&\hat{\bm v}_t = \{ \hat u_{t-1}^{(2)},~\dots,~\hat u_{t-1}^{(\mu)},~v_t \} &\label{algo:CandInput} \\
					&\hat x_{t+\mu} = A^{\mu} x_t + S_c \hat v_t \label{algo:PredStates} &\\
					&x^\pi_{t+\mu} = \Pi_{\bar{\mathcal X}} (\hat x_{t+\mu} - \gamma_x \grad f_{t-1}^x(\hat x_{t+\mu})) &\label{algo:OGDStates} \\
					\intertext{if $\norm{\hat x_{t+\mu}-x^\pi_{t+\mu}} = 0$ set $\alpha_t = 0$, else}
					&\qquad \bar \delta_t = \frac{\delta}{\norm{\hat x_{t+\mu}-x^\pi_{t+\mu}}} &\label{algo:delta} \\
					&\qquad \alpha_t = \frac{1}{1+\bar \delta_t} &\label{algo:alpha}
				\end{flalign}
				\begin{subequations} \label{algo:OCP}
					\begin{flalign}
						&\qquad \text{Find } \bm g_t \in \mathcal U^\mu \text{ such that} &\nonumber \\
						&\quad \qquad x^g(\tau; x_{t}) \in \mathcal X ~ \forall \tau \in \mathbb N_{[t,t+\mu]} &\\
						&\quad \qquad g_t^{(\tau)} \in \mathcal U ~ \forall \tau \in \mathbb N_{[t,t+\mu-1]} &\\
						&\quad \qquad A^\mu x_{t} + S_c g_t = x^\pi_{t+\mu} + \bar \delta_t (x^\pi_{t+\mu} - \hat x_{t+\mu}) &
					\end{flalign}
				\end{subequations}
				\begin{flalign}
					& \hat u_t = (1-\alpha_t) \hat v_t + \alpha_t g_t &\label{algo:PredInputs} \\
					& u_t = \hat u_t^{(1)} & \label{algo:Output}
				\end{flalign}
		}}
	}
	\vspace{-15pt}
\end{algorithm}

Before we state our algorithm, we first define some useful notation. Given an input sequence $\bm u = \{u^{(1)},~u^{(2)},~\dots,~u^{(\mu)}\}$, where $u^{(i)} \in \mathbb R^m$, we denote by $ u = \begin{pmatrix}
(u^{(\mu)})^T & \dots & (u^{(1)})^T
\end{pmatrix}^T$ the vector created by stacking the components of $\bm u$. Moreover, we write $x^u(\tau; x_{t})$ for the state at time $\tau \in [t,~t+\mu]$ when starting at $x^u(t; x_{t}) = x_{t}$ and applying the sequence $\bm u$.

The proposed OCO scheme is given in Algorithm~1. In our framework described above, at every time instant~$t$, Algorithm~1 computes a control input $u_t$ based on the measured state vector $x_{t}$ and the previous cost function $L_{t-1}$. Then, only after applying the control input $u_t$ to system~\eqref{eq:SysDyn}, a new cost function $L_t$ is observed, resulting in the cost $L_t(x_t,u_t)$. Note that the feasibility problem in \eqref{algo:OCP} always has a solution: Since the state $x^\pi_{t+\mu}+\bar \delta_t (x^\pi_{t+\mu} - \hat x_{t+\mu})$ is contained in $\mathcal X$ by the definition of $\bar{\mathcal X}$, Assumption~\ref{assump:contrconst} states that it can be reached from $x_{t} \in \mathcal X$ in $\mu$ time steps.

Roughly speaking, Algorithm~1 predicts the trajectories of system \eqref{eq:SysDyn} and then applies online gradient descent (OGD) (\cite{Zinkevich2003,Hazan2016}) twice to track the optimal input $\eta_t$ and the optimal state $\theta_t$. For that, the proposed algorithm can be separated into three steps. First, OGD is applied in~\eqref{algo:OGDInputs} to compute an estimate $v_t$ of the optimal input $\eta_t$. Second, OGD is applied again to track the optimal state $\theta_t$ in \eqref{algo:CandInput}-\eqref{algo:OGDStates}. Similar to warm-starting in MPC \cite{Rawlings2009}, a candidate input sequence $\hat{\bm v}_t$ for the next $\mu$ time steps is generated by shifting the previously predicted input sequence $\hat{\bm u}_t$ and extending it by $v_t$ in~\eqref{algo:CandInput}. This input sequence is then used to predict the state $\mu$ time steps in the future in~\eqref{algo:PredStates} and OGD is applied again to calculate a desired state $x^\pi_{t+\mu}$ improving the state cost in \eqref{algo:OGDStates}. Last, the predicted input sequence $\hat{\bm u}_{t-1}$ is updated in \eqref{algo:delta}-\eqref{algo:PredInputs} by computing a feasible input sequence $g_t$ in \eqref{algo:OCP}. In contrast to \cite{Nonhoff2020}, Algorithm~1 employs a convex combination of $\hat v_t$ and $g_t$ in \eqref{algo:PredInputs} instead of committing to applying the updated predicted input sequence for the next $\mu$ time steps. Note that application of the whole predicted sequence $\hat{\bm u}_t$ yields $x^{\hat u_t}(t+\mu; x_{t}) = x^\pi_{t+\mu}$ as shown in~\eqref{eq:Predhatueqxpi} in the Appendix. The whole procedure is illustrated in Figure~\ref{fig:algo}.

At every time step $t$, Algorithm~1 solves two projections in \eqref{algo:OGDInputs} and \eqref{algo:OGDStates}. Additionally, a feasibility problem has to be solved in \eqref{algo:OCP} if the predicted state $\hat x_{t+\mu}$ is not optimal. The two projections in \eqref{algo:OGDInputs} and \eqref{algo:OGDStates} are, in general, projections onto convex polytopic sets. In particular, they are computationally cheap if the constraint sets $\mathcal U$ and $\mathcal X$ have a simple shape, such as, e.g., box constraints. The feasibility problem in \eqref{algo:OCP} can be cast as a linear feasibility program since all constraint sets are polytopes resulting in linear constraints.

\begin{figure}
	\centering \small
	\vspace{10pt}
	\tikzstyle{state} = [circle, fill, inner sep = 1pt]

\begin{tikzpicture}[scale=.9]

%\draw[step=.5cm,gray,very thin, dashed] (-3,-2) grid (3,2);
%\node (origin) at (0,0) [] {X};

% constraint set X
\node at (-4.75,-1.5) {$\mathcal X$};
\draw (-4.5,-2.5) -- (-4.5,2.5); 
\draw (-4.5,2.5) -- (3.5,2.5);
\draw (3.5,2.5) -- (3.5,-2.5);
\draw (3.5,-2.5) -- (-4.5,-2.5);

% shrinked constraint set bar X
\draw (-4,-2) -- (-4,2);
\draw (-4,2) -- (3,2);
\draw (3,2) -- (3,-2);
\draw (3,-2) -- (-4,-2);
\node at (2.5,-1.5) {$\bar{\mathcal X}$}; 

\node (xt-1) at (0.25,-1.6) [state, label = below: $x_{t}$] {};
\node (pi-2) at (1.25,2) [state, label = {[label distance = -5pt]20: $x^\pi_{t+\mu-1}$}] {};
\node (hat-1) at (0,3.25) [state, label = {[label distance = -5pt]10: $\hat x_{t+\mu}$}] {}
	edge (pi-2);
\coordinate (grad) at (-1.5,2.56)
	edge [<-,thick,red] node[left,black, near end,xshift=-5pt] {$-\gamma_x \grad f_{t-1}^x(\hat x_{t+\mu})$} (hat-1) ;
\node (pi-1) at (-1.5,2) [state, label = {[label distance = -5pt]170: $x^\pi_{t+\mu}$}] {}
	edge [dashed] (grad);
\node (xf) at (-2, 1.5833) [state] {}
	edge [dashed] (hat-1);
\node at (-2.25,1.3) {$x^\pi_{t+\mu} - \bar\delta_t(x^\pi_{t+\mu} - \hat x_{t+\mu})$};

% Trajectory hat v
\coordinate (v1) at (1.25,-1.5)
	edge (xt-1);
\coordinate (v2) at (2.25,-.75)
	edge (v1);
\coordinate (v3) at (2.5,.25)
	edge (v2);
\coordinate (v4) at (2, 1.25)
	edge node[right] {(a)} (v3)
	edge(pi-2);
	
% Trajectory g
\coordinate (g1) at (1.25,-.75)
	edge (xt-1)
	edge [dashed, thin] (v1);
\coordinate (g2) at (1.5,0)
	edge (g1)
	edge [dashed, thin] (v2);
\coordinate (g3) at (1.25,.75)
	edge node[left] {(b)} (g2) 
	edge [dashed, thin] (v3);
\coordinate (g4) at (0.5, 1.25)
	edge (g3)
	edge [dashed, thin] (v4);
\coordinate (g5) at (-.5,1.6)
	edge (g4)
	edge (xf)
	edge [dashed, thin] (pi-2);
	
% Trajectory u
\coordinate (u1) at (1.25,-.9375)
	edge [very thick, green] (xt-1);
\coordinate (u2) at (1.6875,-0.1875)
	edge [very thick] (u1);
\coordinate (u3) at (1.5626,.625)
	edge [very thick] (u2);
\coordinate (u4) at (0.875, 1.25)
	edge [very thick] node[right] {(c)} (u3);
\coordinate (u5) at (-.0625,1.7)
	edge [very thick] (u4)
	edge [very thick] (pi-1);

\end{tikzpicture}%
	\vspace{-0pt}
	\caption{Schematic illustration of Algorithm~1. First, the predicted input sequence $\hat{\bm v}_t$ (a) is used to compute $\hat x_{t+\mu}$. Then, one gradient descent step (red) is applied. A feasible input sequence $\bm g_t$ (b) and a convex combination of (a) and (b) are computed that lead to the updated predicted input sequence $\hat{\bm u}_t$ (c, bold). Finally, only the first input (green) is applied to the system.}
	\vspace{-10pt}
	\label{fig:algo}
\end{figure}
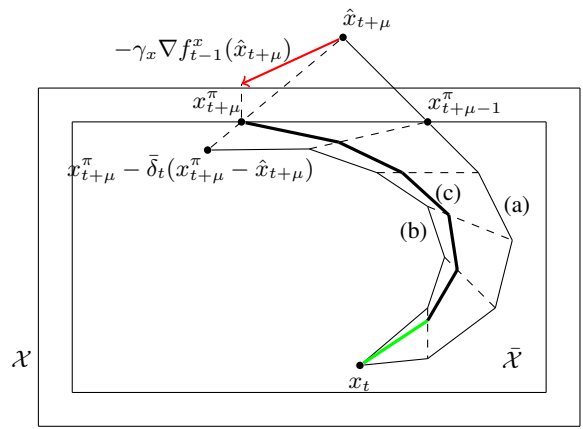

\section{REGRET ANALYSIS} \label{sec:Regret}

In this section, we state our main result, a bound on the regret of Algorithm~1. Its proof is given in the appendix.
\begin{theorem} \label{thm} 
	Let $x_0 \in \mathcal X$ and Assumptions~\ref{assump:constraints}-\ref{assump:contrconst} be satisfied. Given a feasible initialization $\bm{\hat u_0}$, $v_0$, and step sizes $\gamma_u \leq \frac{2}{l_u+\alpha_u}$ and $\frac{\norm{A}-1}{\norm{A}\alpha_x} < \gamma_x \leq \frac{2}{l_x + \alpha_x}$, the dynamic regret $\mathcal R$ of Algorithm~1 can be upper bounded by
	\vspace{-5pt}
	\begin{align*}
	\mathcal R \leq C_0 + C_\theta \sum_{t=0}^T \norm{\theta_t - \theta_{t-1}} + C_\eta \sum_{t=0}^T \norm{\eta_t - \eta_{t-1}},
	\end{align*} 
	for some constants $C_0,~C_\theta,~C_\eta>0$ independent of $T$. Moreover, $x_t \in \mathcal X$ and $u_t \in \mathcal U$ for all $t \in [0,T]$.
\end{theorem}

Theorem \ref{thm} states that the regret of Algorithm~1 is linear in the path length. Hence, we achieve sublinear regret if the path length is sublinear in $T$. This result is well aligned with other results in the literature, see, e.g., \cite{Mokhtari2016,Li2018,Li2019}. Despite the presence of input and state constraints in our setting, we achieve the same sublinear regret bound as in the unconstrained case \cite{Nonhoff2020} up to constant factors. Note that, as already discussed in \cite{Nonhoff2020}, this result implies asymptotic convergence to the optimal equilibrium if it holds that $(\theta_t,\eta_t) = (\theta_{t'},\eta_{t'})$ for some $t' \in \mathbb N$ and all $t \geq t'$. In addition, Theorem~\ref{thm} guarantees constraint satisfaction for every time instant $t \in \mathbb{N}_{[0,T]}$.

\section{SIMULATIONS} \label{sec:Simulations}

In this section, we illustrate the implementation of Algorithm~1 and its performance and compare it to the algorithm presented in \cite{Nonhoff2020}, which cannot guarantee constraint satisfaction. We consider the same system as in \cite{Nonhoff2020} given by
\begin{equation*}
	x_{t+1} = \begin{pmatrix} x_{t+1,1} \\ x_{t+1,2} \\ x_{t+1,3} \end{pmatrix} = \begin{pmatrix} 1.05 & 0.7 & 1.75 \\ 0.35 & 0.7 & 1.05 \\ 1.4 & 0.105 & 1.855 \end{pmatrix} \! x_t + \! \begin{pmatrix} 1 \\ 0 \\ 1 \end{pmatrix} \! u_t
\end{equation*}
and add state and input constraints 
\begin{equation*}
\begin{aligned}
	&x_t \in \mathcal X_0 = \{ x \in \mathbb R^3 | |x_{1}| \leq 3,|x_{2}| \leq 2, |x_{3}| \leq 1 \}, \\
	&u_t \in \mathcal U = \{ u \in \mathbb R |~ |u| \leq 4 \}.
\end{aligned}
\end{equation*}
Unfortunately, the state constraint set $\mathcal X_0$ does not satisfy Assumption~\ref{assump:contrconst}, because it includes states which are either not reachable by a feasible input sequence or outside of the viability kernel. We apply methods from \cite{Maidens2013} and the Multi-Parametric Toolbox \cite{MPT3} to compute a subset~$\mathcal X$ of the viability kernel that satisfies Assumption~\ref{assump:contrconst} with $\mu = 6$. The shrinked constraint set~$\bar{\mathcal X}$ is calculated as detailed in Section~\ref{sec:Setting} with $\delta = 0.01$. Moreover, we choose cost functions $L_t(x,u) = f_t^x(x) + f_t^u(u) = \frac{1}{2}\norm{x-\theta_t}^2 + \frac{1}{2}\norm{u - \eta_t}^2$. The optimal state and input $(\theta_t,\eta_t)$ are time-varying and a priori unknown. The initial condition $x_0$ and initial feasible inputs $v_0$, $\bm{\hat u_0}$ were all set to $0$. The step sizes were chosen as $\gamma_x = 0.98$ and $\gamma_v = 0.98$ for both algorithms, satisfying the assumptions of Theorem~\ref{thm}. The auxiliary input $g_t$ in \eqref{algo:OCP} in Algorithm~1 is found by the 'linprog' command in Matlab.

Figure~\ref{fig:SimResults} shows the simulated closed loop for both algorithms, the optimal states $\theta_t$ and inputs $\eta_t$ and the corresponding constraints. It can be seen that Algorithm~1 tracks the optimal equilibrium. Compared to the algorithm presented in \cite{Nonhoff2020}, Algorithm~1 is slower to react to setpoint changes. This is due to the fact that controllability under constraints in Assumption~\ref{assump:contrconst} requires $\mu = 6$, whereas \mbox{$\mu=n=3$} could be used in \cite{Nonhoff2020} (see Remark~\ref{rem:PredHor}).
On the other hand, in contrast to the algorithm from \cite{Nonhoff2020}, Algorithm~1 is guaranteed to satisfy the constraints at all times.

\begin{figure}
	\centering \small
	\setlength\breite{.4\textwidth}
	\vspace{5pt}
	\input{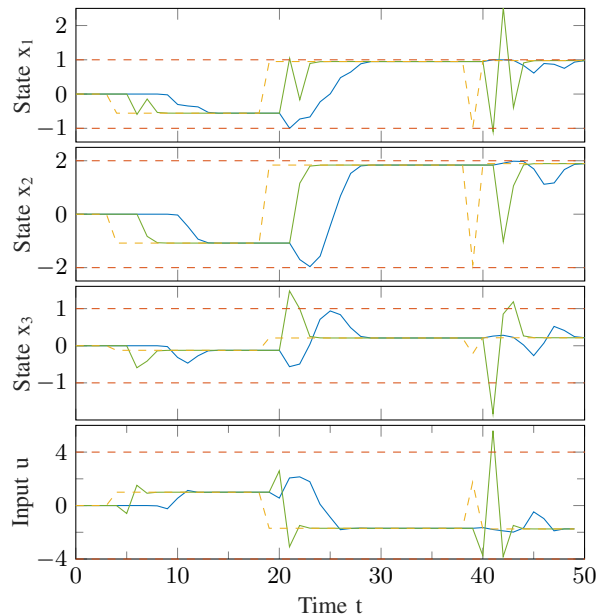}
	\vspace{-5pt}
	\caption{State and input trajectories when applying Algorithm 1 (blue solid) compared to the algorithm presented in \cite{Nonhoff2020} (green solid) and the the optimal states $\theta_t$ and inputs $\eta_t$ (yellow dashed) together with the constraints (red dashed).}
	\label{fig:SimResults}
	\vspace{-18pt}
\end{figure}

\section{CONCLUSION} \label{sec:Conclusion}

In this work, we apply online convex optimization to linear dynamical systems subject to polytopic state and input constraints. We give an online algorithm that achieves sublinear regret if the variation of the cost functions, measured in path length, is sublinear and guarantees constraint satisfaction.

There are two obvious directions for future research. On the one hand, the prediction horizon could be shortened as discussed in Remark~\ref{rem:PredHor}, which may result in the algorithm not being able to apply a full gradient step within the shorter horizon. In this case, new analysis techniques are required to prove that a sublinear regret bound still holds. On the other hand, Assumption \ref{assump:tracking} could be relaxed, allowing economic cost functions. %In addition, predictions on the future cost functions and more efficient online optimization algorithms than OGD could improve the algorithm's performance.

%There are two obvious directions for future research. On the one hand, the prediction horizon could be shortened as discussed in Remark~\ref{rem:PredHor}. On the other hand, Assumption \ref{assump:tracking} could be relaxed, allowing economic cost functions.


%%%%%%%%%%%%%%%%%%%%%%%%%%%%%%%%%%%%%%%%%%%%%%%%%%%%%%%%%%%%%%%%%%%%%%%%%%%%%%%%

%%%%%%%%%%%%%%%%%%%%%%%%%%%%%%%%%%%%%%%%%%%%%%%%%%%%%%%%%%%%%%%%%%%%%%%%%%%%%%%%

%%%%%%%%%%%%%%%%%%%%%%%%%%%%%%%%%%%%%%%%%%%%%%%%%%%%%%%%%%%%%%%%%%%%%%%%%%%%%%%%
\section*{APPENDIX}

Before we prove Theorem~\ref{thm}, we give some auxiliary results. First, since $\theta$ and $\eta_t$ are only defined for $0 \leq t \leq T$, we fix without loss of generality for the remainder of this work $\eta_t = v_0$ and $\theta_t = \hat x_\mu$ for all $t < 0$. 

Second, in order to shorten notation, let $\bar \alpha_t = 1 - \alpha_t$ and 
\[ \bar \alpha_j^i = \begin{cases} \prod_{s=i}^j \bar \alpha_s &\text{if } i < j \\ \bar \alpha_i &\text{if } i=j \\ 1 &\text{if } i>j \end{cases}.\]
We have by the definition of $\alpha_t$ in \eqref{algo:alpha} that $0 \leq \alpha_t < 1$ and, hence, $0 < \bar \alpha_t \leq 1$ as well as $0 < \bar \alpha_{t+s}^t \leq 1$ for any $s \in \mathbb N$. Moreover, since $\alpha_t + \bar \alpha_t =1$, it holds for any $\tau, s \in \mathbb N$ that
\[
	\bar \alpha_{\tau+s}^\tau + \sum_{j=0}^s \bar \alpha_{\tau+s}^{\tau+1+j} \alpha_{\tau+j} = 1. \numberthis \label{eq:ConvexCombRecursion}
\]

Next, we have the following result on the rate of convergence of projected gradient descent \cite{Nesterov2018}. For an $\alpha$-convex and $l$-smooth function $f:\mathcal X \subset \mathbb R^n \rightarrow \mathbb R$ to be minimized, one projected gradient step $x_1 = \Pi_{\mathcal X} (x_0 - \gamma \grad f(x_0))$, where $\gamma \leq \frac{2}{\alpha+l}$ is a step size parameter, satisfies
\begin{align}
	\norm{x_1 - \theta} \leq \kappa \norm{x_0 - \theta}, \label{eq:ContractionGD}
\end{align}
where $\theta = \arg \min_{x \in \mathcal X} f(x)$ and $\kappa = 1-\alpha \gamma$. Accordingly, we define $\kappa_x = 1-\alpha_x \gamma_x$ and $\kappa_u = 1-\alpha_u \gamma_u$.

Third, we examine the closed-loop trajectories and the predicted trajectories of Algorithm~1. Using \eqref{algo:PredInputs} and the definition of $\alpha_t$ in Algorithm~1 it can be shown that
\[
	A^\mu x_{t} + S_c \hat u_t = x^\pi_{t+\mu} \numberthis \label{eq:Predhatueqxpi}
\]
holds in both cases, $\alpha_t = 0$ as well as $\alpha_t \neq 0$. Then, the predicted states $\hat x_{t+\mu}$ can be calculated recursively as follows
\begin{align*}
	\hat x_{t+\mu+1} &\refeq{\eqref{algo:CandInput},\eqref{algo:PredStates}} A^\mu \left( A x_{t} + B \hat u_t^{(1)} \right) + S_c \begin{pmatrix} 0 \\ \hat u_t^{(\mu)} \\ \dots \\ \hat u_t^{(2)} \end{pmatrix} + B v_{t+1} \\
	&\refeq{\eqref{eq:Predhatueqxpi}} Ax^\pi_{t+\mu} + Bv_{t+1}. \numberthis \label{eq:RecPredStates}
\end{align*} 
Moreover, the predicted input $\hat u_t$ can be expressed in terms of previous inputs by
\[
	\hat u_t^{(\mu-s)} \refeq{\eqref{algo:PredInputs}} \bar \alpha_{t}^{t-s} v_{t-s} + \sum_{j=0}^{s} \bar \alpha_{t}^{t+1-j} \alpha_{t-j}g_{t-j}^{(\mu-s+j)}. \numberthis \label{eq:RecPredInputs}
\]
The real state trajectory is $x_{t+\mu} = A^\mu x_{t} + S_c \begin{pmatrix} u_{t+\mu-1} \\ \dots \\ u_t \end{pmatrix}$, where the inputs $u_{t+s}$, $s \in \mathbb N_{[0,t+\mu-1]}$, can be expressed by repeatedly inserting \eqref{algo:PredInputs} by
\begin{equation*}
\begin{aligned}
	u_{t+s} &\refeq{\eqref{algo:Output}} \hat u_{t+s}^{(1)} \refeq{\eqref{algo:PredInputs}} \bar \alpha_{t+s} \hat u_{t+s-1}^{(2)} + \alpha_{t+s} g_{t+s}^{(1)} \\
	&\refeq{\eqref{algo:PredInputs}} \bar \alpha_{t+s}^t \hat u_{t-1}^{(s+2)} + \sum_{j=0}^s \bar \alpha_{t+s}^{t+1+j} \alpha_{t+j} g_{t+j}^{(s-j+1)},
\end{aligned}
\end{equation*}
if $0 \leq s < \mu -1$, and
\begin{equation}
u_{t+s} = \sum_{j=0}^{\mu-1} \left( \bar \alpha_{t+\mu-1}^{t+1+j} \alpha_{t+j} g_{t+j}^{(\mu-j)} \right) + \bar \alpha_{t+\mu-1}^t v_t, \label{eq:RecInputs}
\end{equation}
if $s = \mu-1$. Next, we are ready to state the following lemma, which bounds the cumulative prediction error.
\begin{lemma} \label{lemma}
	Let Assumptions~\ref{assump:constraints}-\ref{assump:contrconst} be satisfied. Given step sizes $\gamma_u \leq \frac{2}{l_u+\alpha_u}$ and $\frac{\norm{A}-1}{\norm{A}\alpha_x}< \gamma_x \leq \frac{2}{l_x + \alpha_x}$, it holds that
	\begin{align*}
		\sum_{t=0}^{T-\mu} \norm{\hat x_{t+\mu} - x_{t+\mu}} \leq \mu C_1 \norm{A} &\sum_{t=0}^{T} \norm{\theta_t - \theta_{t-1}}\\
		+ \frac{\kappa_u \mu}{1-\kappa_u}\left( \norm{S_c} + C_1 \norm{B} \right) &\sum_{t=0}^{T} \norm{\eta_t - \eta_{t-1}},
	\end{align*}
	where $C_1 = \frac{D_u (2\mu-1) (1+\kappa_x) \norm{S_c}}{\delta(1-\norm{A}\kappa_x)}$.
\end{lemma}
\begin{proof}%[Lemma~\ref{lemma}]
	First, note that the step size $\gamma_x$ is well-defined due to the bound on $\norm{A}$ in Assumption~\ref{assump:system}.
	
	Then, by inserting \eqref{eq:ContractionGD} we have
	\[
		\sum_{t=0}^T \norm{v_{t+1} - \eta_t} \refleq{\eqref{eq:ContractionGD}}\kappa_u \sum_{t=0}^T \Big( \norm{v_t - \eta_{t-1}} + \norm{\eta_t - \eta_{t-1}} \Big).
	\]
	Due to $v_0 = \eta_{-1}$ and $1-\kappa_u > 0$ rearranging yields
	\[
		\sum_{t=0}^T \norm{v_{t+1} - \eta_t} \leq \frac{\kappa_u}{1-\kappa_u} \sum_{t=0}^T \norm{\eta_t - \eta_{t-1}}. \numberthis \label{eq:vt+1-etat}
	\]
	Additionally, since $\theta_{-1} = \hat x_\mu$, we have by Assumption~\ref{assump:tracking}
	\begin{align*}
		&\sum_{t=0}^T \norm{\hat x_{t+\mu} - \theta_{t-1}} \leq \sum_{t=0}^{T} \norm{\hat x_{t+\mu+1} - \theta_t} \\
		\refeq{\eqref{eq:RecPredStates}} ~ &\sum_{t=0}^T \norm{Ax^\pi_{t+\mu} + Bv_{t+1} - A\theta_t - B\eta_t} \\
		\begin{split}
		\hspace{5pt}\refleq{\eqref{eq:ContractionGD},\eqref{eq:vt+1-etat}} \hspace{7pt} &\norm{A}\kappa_x \sum_{t=0}^T \norm{\hat x_{t+\mu} - \theta_{t-1}} + \norm{A} \sum_{t=0}^T \norm{\theta_t - \theta_{t-1}}\\
		&\quad + \frac{\kappa_u}{1-\kappa_u}\norm{B} \sum_{t=0}^T \norm{\eta_t - \eta_{t-1}}.
		\end{split}
	\end{align*}
	The lower bound on the step size $\gamma_x$ implies $\norm{A}\kappa_x < 1$. Hence, rearranging yields
	\begin{align}
		\begin{split} \label{eq:hatx-theta}
		\sum_{t=0}^T &\norm{\hat x_{t+\mu} - \theta_{t-1}} \leq \frac{\norm{A}}{1-\norm{A}\kappa_x} \sum_{t=0}^T \norm{\theta_t - \theta_{t-1}} \\&\qquad + \frac{\norm{B}\kappa_u}{(1-\kappa_u)(1-\norm{A}\kappa_x)} \sum_{t=0}^T \norm{\eta_t - \eta_{t-1}}. 
		\end{split}
	\end{align}
	Moreover, using the triangle inequality we get
	\begin{align*}
		&\sum_{t=0}^T \norm{\hat x_{t+\mu} - x^\pi_{t+\mu}} \refleq{\eqref{eq:ContractionGD}} \hspace{3pt} (1+\kappa_x) \sum_{t=0}^T \norm{\hat x_{t+\mu}-\theta_{t-1}} \\
		\begin{split}
		\refleq{\eqref{eq:hatx-theta}}\hspace{3pt}&\frac{\norm{A}(1+\kappa_x)}{1-\norm{A}\kappa_x} \sum_{t=0}^T \norm{\theta_t - \theta_{t-1}} \\
		&\qquad + \frac{\norm{B}\kappa_u(1+\kappa_x)}{(1-\kappa_u)(1-\norm{A}\kappa_x)} \sum_{t=0}^T \norm{\eta_t - \eta_{t-1}}.
		\end{split} \numberthis \label{eq:hatx-xpi}
	\end{align*}
	Last, we combine all the above results to proof Lemma~\ref{lemma}. By \eqref{algo:PredStates}, we get $\hat x_{t+\mu} - x_{t+\mu} \hspace{7pt} = \hspace{7pt} S_c \begin{pmatrix} 
		v_t - u_{t+\mu-1} \\ \hat u_{t-1}^{(\mu)} - u_{t+\mu-2} \\ \vdots \\ \hat u_{t-1}^{(2)} - u_t \end{pmatrix}$,	where the right-hand side of the equation can be rewritten using \eqref{eq:RecInputs}. By adding $S_c\begin{pmatrix} \eta_{t-1}^T & \dots & \eta_{t-\mu}^T \end{pmatrix}^T-S_c\begin{pmatrix} \eta_{t-1}^T & \dots & \eta_{t-\mu}^T \end{pmatrix}^T$, taking the norm on both sides, using \eqref{eq:ConvexCombRecursion} and \eqref{eq:RecPredInputs}, we arrive at
	\[
		\norm{\hat x_{t+\mu} {-} x_{t+\mu}} {\leq} \norm{S_c} \! \left( \sum_{i=0}^{\mu-1} \! \norm{ v_{t-i} {-} \eta_{t-1-i} } {+} \mu D_u \hspace{-1.5ex} \sum_{j=1-\mu}^{\mu-1} \hspace{-1.5ex} \alpha_{t+j} \! \right)
	\]
	By summing over $t$ on both sides we get
	\begin{equation*}
	\begin{split}
		\sum_{t=0}^{T-\mu} &\norm{\hat x_{t+\mu} - x_{t+\mu}} \refleq{\eqref{algo:alpha}} \norm{S_c} \mu \sum_{t=0}^{T-\mu} \norm{v_{t+1} - \eta_{t}} \\ &+  \norm{S_c} \mu D_u (2\mu-1) \frac{1}{\delta} \sum_{t=0}^{T} \norm{\hat x_{t+\mu} - x^\pi_{t+\mu}},
	\end{split}
	\end{equation*}
	It remains to insert \eqref{eq:vt+1-etat} and \eqref{eq:hatx-xpi} to get the result.
\end{proof}
Now, we are finally ready to proof Theorem~\ref{thm}. \\
\begin{proof}
	First, we show the regret bound for Algorithm~1 and then discuss feasibility of the states and inputs. In order to obtain an upper bound for the regret, we begin by bounding the suboptimality of the chosen control inputs. By the definition of $u_t$ in \eqref{algo:Output} we have
	\begin{align*}
		&\sum_{t=0}^T \norm{u_t - \eta_t} = \sum_{t=0}^T \norm{\hat u_t^{(1)} - \eta_t } \\
		\begin{split}
		\refleq{\eqref{eq:ConvexCombRecursion},\eqref{eq:RecPredInputs}} \hspace{9pt} &\sum_{t=0}^T \norm{\bar \alpha_{t}^{t-\mu+1}(v_{t-\mu+1}-\eta_{t-\mu})} + \sum_{t=0}^T \norm{\eta_t - \eta_{t-\mu}} \\
		+ &\sum_{t=0}^T \norm{\sum_{j=0}^{\mu-1} \bar \alpha_{t}^{t+1-j} \alpha_{t-j} \left( g_{t-j}^{(j+1)} - \eta_{t-\mu} \right)}
		\end{split} \\
		\begin{split}
		\leq\hspace{9pt}&\sum_{t=0}^T \norm{v_{t+1} - \eta_t} + \mu \sum_{t=0}^T \norm{\eta_t - \eta_{t-1}} \\
		+ &\sum_{t=0}^T \sum_{j=0}^{\mu-1} \alpha_{t-j}\norm{g_{t-j}^{(j+1)}-\eta_{t-\mu}},
		\end{split}
	\end{align*}
	where we threw away terms $0<\bar \alpha^i_j \leq 1$ and used a telescoping series and the triangle inequality in the last line. By \eqref{eq:vt+1-etat} and the definition of $\alpha_t$ in \eqref{algo:alpha}, it holds that
	\[
		\sum_{t=0}^T \norm{u_t - \eta_t} \refleq{\eqref{eq:hatx-xpi}} C_2 \sum_{t=0}^T \norm{\eta_t - \eta_{t-1}} + C_3 \sum_{t=0}^T \norm{\theta_t - \theta_{t-1}}, \numberthis \label{eq:ut-etat}
	\]
	where $C_2 = \frac{\kappa_u\norm{B}(1+\kappa_x)\mu D_u}{\delta(1-\kappa_u)(1-\norm{A}\kappa_x)} + \frac{\kappa_u}{1-\kappa_u} + \mu$ and \linebreak $C_3 = \frac{\norm(A)(1+\kappa_x)\mu D_u}{\delta(1-\kappa_x)}$. %Next, applying the triangle inequality yields
%	\begin{align*}
%		\sum_{t=0}^k \norm{\theta_{t+p}-\theta_{t-1}} &\leq \sum_{t=0}^k \sum_{j=0}^p \norm{\theta_{t+j} - \theta_{t+j-1}} \\
%		&\leq (p+1) \sum_{t=0}^{k+p} \norm{\theta_t - \theta_{t-1}}. \numberthis \label{eq:thetap-theta}
%	\end{align*}
	Last, we bound the regret $\mathcal R$ of Algorithm~1. Optimality of $\theta_t$ and $\eta_t$ implies
	\begin{align*}
		&\mathcal{R} \leq \sum_{t=0}^T f_t^x(x_t)+f_t^u(u_t)-f_t^x(\theta_t)-f_t^u(\eta_t) \\
		&\leq L_x \! \sum_{t=0}^{\mu-1} \! \norm{x_t {-} \theta_t} {+} L_x \! \sum_{t=0}^{T-\mu} \! \norm{x_{t+\mu} {-} \theta_{t+\mu}} {+} L_u \! \sum_{t=0}^T \! \norm{u_t {-} \eta_t}.
	\end{align*}
	Compactness of the set $\mathcal X$ implies existence of a finite constant $D_x$ that satisfies $\norm{x-y} \leq D_x$ for all $x,y \in \mathcal X$. Since $x_t,\theta_t \in \mathcal X$ and by \eqref{eq:ut-etat}, we obtain
	\begin{align*}
		\begin{split}
		\mathcal{R} &\refleq{\eqref{eq:ut-etat}} L_x \mu D_x + L_x \sum_{t=0}^{T-\mu} \norm{\hat x_{t+\mu} - x_{t+\mu}} + L_u \sum_{t=0}^T \norm{u_t - \eta_t} \\
		& + L_x \sum_{t=0}^{T-\mu} \norm{\hat x_{t+\mu} - \theta_{t-1}} + L_x \sum_{t=0}^{T-\mu} \norm{\theta_{t+\mu} - \theta_{t-1}}.
		\end{split}
	\end{align*}
	Inserting Lemma \ref{lemma}, \eqref{eq:hatx-theta}, and \eqref{eq:ut-etat} yields the desired result. The constants are given by $C_0 = L_x \mu D_x$, $C_\eta = L_u C_2 + \frac{L_x\norm{B}\kappa_u}{(1-\norm{A}\kappa_x)(1-\kappa_u)} + \frac{L_x\kappa_u\mu}{1-\kappa_u}(\norm{S_c}+C_1\norm{B})$ and $C_\theta = L_x\mu C_1\norm{A} + \frac{L_x \norm{A}}{1-\norm{A}\kappa_x} + L_x(\mu+1)+L_uC_3$. 
	
	%Last, we show feasibility of the state and input trajectories emerging from application of Algorithm~1 by induction. In the following, we assume that $\mathbf{\hat u}_{t-1}$ was a feasible input sequence with respect to the state constraints at time $t-1$, i.e., $x^{\hat u_{t-1}}(\tau,x_{t-1}) \in \mathcal X$ for all $\tau \in [t,~ t+\mu-1]$. Thus, we have that $x^{\hat v_t}(\tau,x_{t}) \in \mathcal X$ for all $\tau \in [t,~t+\mu-1]$, i.e., $\hat{\mathbf v}_t$ is a feasible input sequence for all but possibly the last time step. By definition of $g_t$ in \eqref{algo:OCP}, $g_t$ is a feasible input sequence, too, and we have that $x^{g_t}(\tau;x_{t}) \in \mathcal X$ for all $\tau \in [t,~t+\mu-1]$. Since $\hat{\bm u}_t$ is a convex combination of $\hat{\bm v}_t$ and $\bm g_t$ by definition of $\hat{\bm u}_t$ in \eqref{algo:PredInputs}, we obtain $x^{\hat u_t}(\tau;x_{t}) \in \mathcal X$ for all $\tau \in [t,~t+\mu-1]$ by convexity of $\mathcal X$. Moreover, we have that $x^{\hat u_t}(t+\mu;x_{t}) \refeq{\eqref{eq:Predhatueqxpi}} x^\pi_{t+\mu}$. Since $x^\pi_{t+\mu} \in \mathcal{\bar X} \subset \mathcal X$ by \eqref{algo:OGDStates}, we have shown that $\hat{\bm u}_t$ is a feasible input sequence with respect to the state constraints, which implies $x_t \in \mathcal X$. The result then follows by induction, because $\hat{\mathbf u}_0$ admits a feasible initial input sequence at time $t=0$. Feasibility with respect to the input constraints follows by similar arguments.
	
	Last, we show feasibility of the state and input trajectories emerging from application of Algorithm~1. Due to space restrictions, we only give a brief sketch of the proof. We assume that $\mathbf{\hat u}_{t-1}$ was a feasible input sequence at time $t-1$. Thus, we have that $\hat{\mathbf v}_t$ is a feasible input sequence for all but possibly the last time step. Moreover, $g_t$ is a feasible input sequence by construction. Since $\hat{\bm u}_t$ is a convex combination of $\hat{\bm v}_t$ and $\bm g_t$, we obtain that $\hat{\bm u}_t$ is feasible for all but the last time step by convexity of $\mathcal X$ and $\mathcal U$. Moreover, we have that $x^{\hat u_t}(t+\mu;x_{t}) \refeq{\eqref{eq:Predhatueqxpi}} x^\pi_{t+\mu} \in \mathcal{\bar X} \subset \mathcal X$ and $v_t \in \mathcal U$, which shows that $\hat{\bm u}_t$ is a feasible input sequence. The result then follows by induction, because $\hat{\mathbf u}_0$ admits a feasible initialization.
\end{proof}

%%%%%%%%%%%%%%%%%%%%%%%%%%%%%%%%%%%%%%%%%%%%%%%%%%%%%%%%%%%%%%%%%%%%%%%%%%%%%%%%

\bibliographystyle{ieeetr}
\bibliography{biblio}

%\printbibliography

\end{document}